\documentclass{article}[12pt]
\usepackage{graphicx}
\usepackage{setspace}
 \DeclareGraphicsExtensions{.eps, .ps}
\usepackage{soul,color}
\usepackage{amsmath}
\usepackage{amsthm}
\usepackage{amsfonts}
\usepackage{newfloat}
\DeclareFloatingEnvironment[name={Supplementary Figure}]{suppfigure}

\def\bft{{\bf t}}
\def\pr{\mathop{\rm pr}\nolimits}
\def\cov{\mathop{\rm cov}\nolimits}
\def\var{\mathop{\rm var}\nolimits}

\def\Ga{\mathop{\rm Ga}\nolimits}

\def\bfX{{\bf X}}

\def\indep{\mathrel{\rlap{$\perp$}\kern1.6pt\mathord{\perp}}}
\def\Real{\mathbb{R}}
\def\binomial#1#2{\left(\begin{array}{c} {#1} \\ {#2} \end{array} \right)}
\def\ascf#1#2{{#1}^{\uparrow #2}}

\def\given{\,|\,}

\def\Rsharp{R^\sharp}
\def\Bsharp{B^\sharp}

\def\P{{\cal P}}
\def\S{{S}}
\def\PR{{\cal OP}}
\def\OP{{\cal OP}}

\def\beginignoretext{\setbox0=\vbox\bgroup}
\def\endignoretext{\egroup}

\numberwithin{equation}{section}
\theoremstyle{plain}
\newtheorem{theorem}{Theorem}[section]
\newtheorem{example}{Example}

\newtheorem{definition}[theorem]{Definition}

\newtheorem{lemma}[theorem]{Lemma}
\newtheorem{proposition}[theorem]{Proposition}

\begin{document}
\title{Weak continuity of predictive distribution for \\ Markov survival processes}
\author{Walter Dempsey and Peter McCullagh
\thanks{Department of Statistics,
University of Chicago, 5734
University Ave, Chicago, Il 60637, U.S.A.}}
\maketitle
\begin{abstract}\noindent
We explore the concept of a consistent exchangeable survival process - 
a joint distribution of survival times in which the risk set evolves as a 
continuous-time Markov process with homogeneous transition rates.  We show 
a correspondence with the de Finetti approach of constructing an exchangeable 
survival process by generating iid survival times conditional on a completely 
independent hazard measure.  We describe several specific processes, showing 
how the number of blocks of tied failure times grows asymptotically with the 
number of individuals in each case.  
In particular, we show that the set of Markov survival processes 
with weakly continuous predictive distributions can be characterized by
a two-dimensional family called the harmonic process.
We end by applying these methods to data, 
showing how they can be easily extended to handle censoring.
\end{abstract}

\section{Introduction}
\subsection{Background}
This paper introduces the concept of a consistent exchangeable
survival process, which is a joint distribution $p_n$ of survival times
$T_1,\ldots, T_n$ such that the temporal progress of the
risk set is Markovian with homogeneous transition rates.
The approach is similar to that taken by
Kingman (1980) or Aldous (1996) in constructing coalescent-like trees,
or by McCullagh, Pitman and Winkel (2006)
in characterizing self-similar Gibbs fragmentation trees.
Each consistent survival process is generated by a splitting rule 
which determines the joint distribution 
for fixed $n$ and the predictive distributions,
$\pr(T_{n+1} \in A \given T_1,\ldots, T_n)$.
We show that each splitting rule is defined by a corresponding
characteristic index, $\zeta = \{ \zeta_i \}_{i=1}^\infty$. 
In particular, the marginal distributions are exponential with rate~$\zeta_1$.

Three closely related processes in the class are explored in detail,
the power process for which $\zeta_n \propto n^\rho$ for $0 < \rho < 1$,
the gamma process for which $\zeta_n \propto \log(1+n/\rho)$ and
the harmonic process for which $\zeta_n \propto \sum_{j=0}^{n-1} 1/(\rho+j)$.
The latter has a simple form for the joint density,
it is easy to generate sequentially, and
the conditional distributions have a close affinity with
various long-standing estimators for the survival distribution,
such as the Kaplan-Meier estimator (Kaplan and Meier, 1958).
Moreover, the harmonic process is shown to be the \emph{only} 
Markov survival process with weakly continuous predictive distributions.

The de~Finetti approach to constructing an exchangeable survival process
is to generate survival times conditionally independent and identically
distributed given a completely independent hazard measure,
i.e.~the cumulative conditional hazard is a L\'evy process
(Cornfield and Detre, 1977; Kalbfleisch, 1978; Hjort, 1990; Clayton, 1991).
Such processes are sometimes called \emph{neutral to the right}
(Doksum, 1974; James, 2006).
Each process in this class is determined by the characteristic exponent
of the L\'evy measure, and the temporal evolution of the risk set
is Markovian, i.e.,~they are exchangeable Markov survival processes.
Thus, the two approaches, which are mathematically natural in very
different ways, give rise to the same set of exchangeable processes.
The probability density can be computed using the characteristic index,
and the process can be simulated directly from the predictive distributions,
by-passing the L\'evy process entirely.

A survival process in this class typically has multiple tied failure times,
and these ties generate an exchangeable partial ranking,
or an ordered partition, of the particles.
The partial ranking is infinitely exchangeable, and can be generated by a 
sequential assignment rule similar to the Chinese restaurant process
for the Dirichlet partition
(Ferguson, 1973, 1974; Pitman, 2006).
As a function of $n$, the number of blocks may be bounded,
or it may grow at a logarithmic or super-logarithmic rate,
or it may grow at rate $n^\alpha$ for some $0<\alpha \le 1$.
For the harmonic and the gamma processes, the rate is 
asymptotically proportional to $\log^2 n$.


\subsection{Risk set trajectory}
Consider a set of patients or particles in which particle~$i$ survives
to time $T_i > 0$, not necessarily finite.
To each particle there corresponds an observation interval $[0, c_i]$,
and the event time recorded for particle~$i$ is $Y_i = \min(T_i, c_i)$.
In all of the mathematical theory that follows, the censoring times
are arbitrary positive numbers and are known for each particle, so if $Y_i = c_i$
the event is known to be a censoring time;
otherwise, if $Y_i < c_i$, the event is known to be a death or failure.
Evidently, setting $c_i=0$ is equivalent to removing particle~$i$
from the sample.

It is helpful mathematically to think of the observation on particle~$i$
not as a real number but as a survival interval
\begin{eqnarray*}
R_i &=& [0, Y_i) \quad\hbox{if}\quad Y_i < c_i\\
R_i &=& [0, Y_i] \quad\hbox{if}\quad Y_i = c_i.
\end{eqnarray*}
Alternatively, and equivalently, the observation is a Boolean function such that
$R_i(t) = 1$ if $t \in R_i$  and zero otherwise.
With this notation
\[
R(t) = \{i : R_i(t) = 1\}
\]
is the subset of particles known to be alive at time~$t$,
i.e.,~the risk set at time~$t$.
The temporal trajectory of $R$ is a non-increasing sequence of subsets
beginning at $R(0) = [n] = \{ 1, \ldots, n \}$, which is the sample of particles initially under observation.
The right increment $R(t) \setminus R(t^+)$ is the
subset of particles censored at time~$t$;
the left increment $B(t)=R(t^-) \setminus R(t)$ is the
subset of particles observed to fail at time~$t$.
It is worth emphasizing that the subsets $B(t)$ and $R(t)$ 
are disjoint for every~$t$.

The number of individuals in the risk set $\Rsharp(t) = \#R(t)$
is a non-increasing real-valued function, beginning at $\Rsharp(0) = n$,
decreasing in integral steps at each censoring and failure time
to $\Rsharp(t) = 0$ for $t > \max(Y)$.
The right increment $\Rsharp(t) - \Rsharp(t^+)$ counts censored particles;
the left increment $\Rsharp(t^-) - \Rsharp(t)$ counts failures,
and is regarded as a right-continuous counting measure.


This paper is concerned with survival processes,
by which is meant probabilistic models for the temporal
evolution of the risk set, i.e.,~stochastic models in continuous time
with state space equal to the subsets of~$[n]$.
Since $n$~is arbitrary, each stochastic process must be defined
in a consistent manner for general samples in such a way that
$p_n$ is the marginal distribution of $p_{n+1}$ under deletion,
either by setting $c_{n+1} = 0$ or by removal from the sample.
Censoring is not ignored, but the discussion presumes $c_i=\infty$
to avoid unnecessary distractions.
Although it affects the evolution of the risk set, censoring has a
relatively trivial effect on all probability calculations
and associated statistical procedures such as
parameter estimation and the calculation of conditional distributions
such as the survival function $\pr(T_{n+1} > t \given R[n])$ or the residual
survival function $\pr(T_{n+1} >t \given R[n], T_{n+1} > s)$, 
where $R[n]$ denotes the risk set trajectory restricted to the 
set of particles, $[n]$.
The natural convention on left and right increments is consistent with
standard practice (Kalbfleisch, 1978, Table~1), and is sufficient
for procedures to be expressed in the generality required to
accommodate arbitrary fixed right censoring.

\section{Exchangeable Markov survival processes}
\subsection{Partial rankings}
A partial ranking of $[n]$ is an \emph{ordered list} $B=(B_1,\ldots, B_k)$
consisting of disjoint non-empty subsets of $[n]$ whose union is~$[n]$.
The terms \emph{partial ranking} and \emph{ordered partition} are used interchangeably.
The elements of $[n]$ are unordered within blocks, but $B_1$ is the subset
ranked first, $B_2$ is the subset ranked second, and so on.
The alternative notation $B=B_1|B_2|\cdots|B_k$ is sometimes used,
either for ordered or unordered partitions.

There are three partial rankings of $[2]$, 12 of $[3]$, and 75 of $[4]$,
arranged in equivalence classes (with respect to permutation of elements
and permutation of unequal-sized blocks) as follows;
\begin{eqnarray*}
\OP_2:&& 12;\quad 1|2,\; 2|1 \cr
\OP_3:&& 123;\quad 12|3, 13|2, 23|1,\; 1|23, 2|13, 3|12;\quad 1|2|3 (3!)\cr
\OP_4:&& 1234;\quad 123|4(4\times2);\quad 12|34(6);\quad  12|3|4(12\times3);\quad 1|2|3|4(4!)
\end{eqnarray*}
Thus $12|3|4(12\times3)$ means that there are three equivalence classes of types
$211, 121, 112$ depending on whether the first, second or third block contains two elements,
and that each class contains 12 partial rankings generated by permutation of elements.

To each unordered partition of $[n]$ containing $k$ blocks,
there correspond $k!$ partial rankings,
one for each ordering of the blocks.
Thus, if $\S_{n,k}$ is number of partitions of $[n]$ containing $k$ blocks,
i.e.,~Stirling's number of the second kind,
then
\[
\#\P_n = \sum_{k=1}^n \S_{n,k},\qquad
\#\PR_n = \sum_{k=1}^n k!\, \S_{nk}
\]
is the number of partitions of $[n]$ and
the number of partial rankings of~$[n]$ respectively.
The numerical values for $n\le 10$ are as follows:
\[
\arraycolsep=4pt
\begin{array}{rrrrrrrrrrr}
n&1&2&3&4&5&6&7&8&9&10\cr
\#\P_n& 1& 2& 5& 15& 52& 203& 877& 4\,140& 21\,147& 115\,975 \cr 
\#\OP_n& 1& 3& 13& 75& 541& 4\,683& 47\,293& 545\,835& 7\,087\,261& 102\,247\,563\cr
\end{array}
\]

\subsection{Exchangeable Markov partial rankings}
A random partial ranking of $[n]$ is a probability distribution $p_n(\cdot)$ on
the finite set $\PR_n$.
The distribution is \emph{finitely exchangeable} if $p_n(\cdot)$
depends only on the block sizes and the block order:
in general $p_n(B_1,\ldots, B_k) \neq p_n(B_{\sigma(1)},\ldots, B_{\sigma(k)})$
for a permutation $\sigma$ of the blocks.

\begin{definition}[Markovian] 
The distribution $p_n$ is said to be \emph{Markovian} if
for each subset of size $1\le m \le n$ there exists
a splitting distribution $q(s, d)$ with $s+d = m$ and $d \ge 1$ such that
the partial ranking $B=(B_1,\ldots,B_k)$ of $[n]$
consisting of $k$ blocks of sizes $\Bsharp_1,\ldots ,\Bsharp_k$ has probability
\begin{eqnarray*}
p_n(B) &=& q(n-\Bsharp_1, \Bsharp_1) \times p_{n-\Bsharp_1}(B_2,\ldots,B_k) \cr
	&=& \prod_{j=1}^k  q(n-\Bsharp_1-\cdots-\Bsharp_{j}, \Bsharp_j).
\end{eqnarray*}
\end{definition}

The following proposition is a direct consequence of the above definition.

\begin{proposition}
Every non-negative function, $q(r,s)>0$ for $s>0$, subject to normalization
\begin{equation}\label{sumconstraint}
\sum_{d=1}^n \binomial n d q(n-d, d) = 1
\end{equation}
for all $n \geq 0$ defines a Markovian survival process.
\end{proposition}

In the context of survival models,
the first argument of $q$ is the number of survivors,
and the second is the number of failures at each successive hazard,
so $q(s, d) \neq q(d, s)$ is not symmetric in its arguments.
In this sense, the splitting distribution for partial rankings is very different
from the splitting distribution for Markov fragmentation trees
(Aldous 1996; McCullagh, Pitman and Winkel 2006).
The normalization conditions are also different.

\subsection{Kolmogorov consistency}
A sequence $p = (p_1, p_2,\ldots)$ in which $p_n$ is a probability distribution on
partial rankings of $[n]$ is said to be Kolmogorov-consistent if
each $p_n$ is the marginal distribution of $p_{n+1}$ under
the map $\PR_{n+1}\to\PR_n$ in which the element $n+1$ is ignored or deleted.
This condition implies that the censoring time for one particle, $c_{n+1}$, for example,
has no effect on the survival times for other particles.
The concern here is the conditions for consistency of exchangeable Markov partial rankings,
so the choice of element to be deleted is immaterial.
Without consistency, the sequence $\{p_n\}$ does not determine a process,
there is no associated partial ranking of the infinite set (population),
and there is no possibility of inference using conditional distributions.

Consider the family of Markov partial rankings generated by
a splitting distribution $q(\cdot, \cdot)$.  Proposition~\ref{consistencycondition_prop}
details when such a splitting distribution corresponds with a Kolmogorov-consistent
Markov survival process.

\begin{proposition}[Consistency condition] \label{consistencycondition_prop}
The splitting distribution corresponds to a Kolmogorov-consistent
Markov survival process if and only if
\begin{equation}\label{consistencycondition_eq}
(1 - q(n, 1)) q(n-d, d) = q(n-d, d+1) + q(n-d+1, d)
\end{equation}
for all integers $n\ge d\ge 1$.  In particular, equation~(\ref{consistencycondition_eq})
implies that the splitting probabilities are determined by the sequence of singleton
splits $q(n,1)$ for $n \geq 0$.
\end{proposition}

\begin{proof}
To derive the conditions for consistency proceed by induction
supposing that $p_1,\ldots, p_n$ are mutually consistent in the Kolmogorov sense.
Then, in order for $p_{n+1}$ to be consistent with $p_n$, it must be
for each ordered partition $B$ of $[n]$,
\begin{eqnarray*}
p_n(B) = q(n, 1) p_n(B) &+& q(n-\Bsharp_1, \Bsharp_1+1) p_{n-\Bsharp_1}(B_2,\ldots) \cr
 &+& q(n-\Bsharp_1+1, \Bsharp_1) p_{n-\Bsharp_1+1}(\cdots).
\end{eqnarray*}
The terms appearing on the right are all of the events in $\PR_{n+1}$
such that deletion of $n+1$ gives rise to the ordered partition $B\in\PR_n$.
Either $n+1$ occurs in the first block as a singleton,
which has probability $q(n, 1) p_n(B)$;
or it occurs appended to $B_1$ as a non-singleton,
which has probability $q(n-\Bsharp_1+1, \Bsharp_1+1) p_{n-\Bsharp_1}(B_2,\ldots)$;
or it occurs elsewhere either as a singleton or 
appended to one of the blocks of $B$,
which occurs with probability $q(n-\Bsharp_1+1, \Bsharp) p_{n-\Bsharp_1+1}(\cdots)$.
But, by the induction hypothesis that $p_1,\ldots, p_n$ are mutually consistent,
the latter event occurs with probability $p_{n-\Bsharp_1}(B_2,\ldots)$.
Hence, an exchangeable Markov family of distributions is Kolmogorov-consistent
if and only if,
for each $n \ge 1$ and each ordered partition $B$ of $[n]$,
\[
(1 - q(n,1)) q(n-\Bsharp_1, \Bsharp_1) =
	q(n-\Bsharp_1, \Bsharp_1+1) + q(n-\Bsharp_1+1, \Bsharp) .
\]
which is equivalent to equation~(\ref{consistencycondition_eq}).
\end{proof}

If the splitting rule corresponds to a Kolmogorov-consistent Markov survival process, it is
said to be a {\it consistent splitting rule}.  A nice property of such rules is that they admit
simple integral representations.

\begin{proposition} \label{int_rep_prop}
Every consistent splitting rule admits an integral representation
\begin{equation} \label{integral_representation}
q(n-d, d) = q_n(d)= \frac{1}{Z_n} \left( \int_0^1 x^{n-d}(1-x)^{d} \varpi (dx) + c \cdot \mathbf{1}_{\{d=1\}} \right) 
\end{equation}
where $Z_n=\int (1-x^n) \varpi (dx) + n \cdot c$.  Here the measure is defined on $[0,1)$ and satisfies 
\[
\int_{0}^1 (1-x) \, \varpi(dx)<\infty.
\]  
The constant, $c$, is called the erosion coefficient, and $\varpi$, is called the dislocation measure.  
\end{proposition}

Both the condition and integral representation are similar to (4) in 
McCullagh, Pitman and Winkel (2008)
and to proposition~41 of Ford (2006).
However, the splitting probabilities in these papers are
symmetric and subject to a normalization condition different from (\ref{sumconstraint}),
so the probabilities are different.  

\subsection{The characteristic index}

Here the {\it characteristic index} associated with each consistent exchangeable partial ranking
is defined and a relation to the splitting rule via $k$th order differences is shown.

\begin{definition} \normalfont
A \emph{characteristic index}, $\zeta$, is a sequence
$\zeta_0,\zeta_1,\zeta_2,\ldots$ beginning with $\zeta_0=0$, $\zeta_1 > 0$, and subsequently
\begin{equation}\label{zeta}
\zeta_{n+1} = \frac{\zeta_n} {1- q_{n,1}} = \zeta_1 \prod_{j=1}^{n} (1 - q_{j,1})^{-1}
\end{equation}
for $n\ge 1$.
\end{definition}
The sequence is said to be in \emph{standardized form} if $\zeta_1 = 1$,
and each standardized sequence determines a Markov partial ranking
provided that the splitting probabilities in (\ref{consistencycondition_eq}) are non-negative.

A natural question is whether the splitting rules can be reconstructed 
from the characteristic index.
Proposition~\ref{char_index_prop} provides an affirmative answer
via $k$th order forward differences, $\Delta^k \zeta$, defined as
\[
(\Delta^k \zeta)_n = \sum_{j=0}^k (-1)^{k-j} \binomial k j \zeta_{n+j},
\]
for $n\ge 0$,
so that $\Delta\zeta_n = \zeta_{n+1} - \zeta_n$ is the first difference,
$(\Delta^2\zeta)_n = \zeta_{n+2} - 2\zeta_{n+1} + \zeta_n$ is the second, and so on.

\begin{proposition} \label{char_index_prop}
Given a characteristic index, $\zeta$, the corresponding splitting rule is given by
\begin{equation}\label{zetacondition}
q(r, d) = \frac{(-1)^{d-1} (\Delta^d \zeta)_r } { \zeta_{r+d}}
\end{equation}
for $r \ge 0$ and $d \ge 1$.
The splitting probabilities are therefore determined by the standardized
sequence $\zeta/\zeta_1$.
Moreover, the characteristic index yields a consistent splitting
rule if it satisfies a non-negativity condition on forward differences; that is,
$(-1)^{d-1} (\Delta^d \zeta)_r \geq 0$.
\end{proposition}

The proof of Proposition~\ref{char_index_prop} can be found in Appendix~\ref{app:char_index_proof}.
The main advantage of working with the characteristic index is that
the Kolmogorov consistency condition (\ref{consistencycondition_eq})
may be written in a more convenient form as 
a non-negativity condition on forward differences.

\subsubsection{Examples}

Since each Markov partial ranking is associated with a
characteristic sequence (modulo scalar multiplication),
the space of Markov partial rankings may be identified with the
space of non-negative sequences satisfying (\ref{zetacondition}).
Evidently this space is a convex cone, closed under positive linear combinations.
It contains the trivial sequence $\zeta_n =n$, for which $\Delta\zeta = (1,1,\ldots)$
and $\Delta^d \zeta \equiv 0$ for $d\ge 2$.
This means $q(r,1) = 1/(r+1)$, so all splits are singletons,
and there are no ties in the partial ranking.

For $\rho \ge 0$, the singleton split $q(n-1,1) = 1/(n+\rho)$ for $n \ge 2$
implies $\zeta_n \propto n+\rho$, linear in~$n$ for $n \ge 1$, with $\zeta_0=0$.
The forward difference sequences imply
that $\Delta\zeta_{n-1} / \zeta_{n} = 1/(n+\rho)$ for $n\ge 2$,
$q(0,n) = (-1)^{n-1} (\Delta^n \zeta)_0 = \rho/(n+\rho)$ and $q(s,d) = 0$ otherwise.
At each event time, one individual out of $n$ fails with probability $1/(n+\rho)$ (for each individual);
otherwise, with probability $\rho/(n+\rho)$, the entire population fails.

It also contains the sequences $\zeta_n = n^\alpha$ for $0 < \alpha < 1$,
for which $q(r,1) = 1 - (1 + 1/r)^{-\alpha}$.  
Similarly, $q(n-1, 1) = 1/n^2$ gives rise to $\zeta_n \propto n/(n+1)$,
and the uniform splitting rule $q(n-d, d) = (n-d)!\, d! / n\, n!$.

For each $0 < \alpha < 1$, the non-decreasing geometric sequence 
$\xi_n = 1 - \alpha^n$ is such that
\[
(-1)^{d-1} (\Delta^d \xi)_r = (-1)^d \alpha^r (\alpha-1)^d = \alpha^r (1-\alpha)^d > 0,
\]
implying that there exists a Markov partial ranking with
binomial splitting probabilities
$q(r,d) = \alpha^r(1-\alpha)^d / \xi_{r+d}$ for $d \ge 1$.
Moreover, for any non-negative measure $w$ on $[0,1)$, the integrated sequence
$\zeta_n = \int_0^1 (1-\alpha^n) \, dw(\alpha)$ is monotone, and is such that
\[
(-1)^{d-1} (\Delta^d\zeta)_r = \int_0^1 \alpha^r (1-\alpha)^d\, dw(\alpha) \ge 0.
\]
Thus, the Kolmogorov condition is automatically satisfied for all measures $w$ such that
$\zeta_1 = \int_0^1 (1-\alpha)\, dw(\alpha)$ is positive and finite.  By equation~\eqref{integral_representation}, this is the complete set of splitting rules with zero erosion measure ($c = 0$).
Finally, for each $\beta > 0$, the sequence
$\zeta_{\beta n} = \int_0^1 (1 - \alpha^{\beta n}) \, dw(\alpha)$
also satisfies the Kolmogorov condition,
so each $\zeta$ having an integral representation may
be extended to a function on the positive real line.

\begin{example}[Beta-splitting rules]
A natural family of measures are those conjugate to the binomial splitting probabilities, 
$dw(\alpha)  = \alpha^{\rho -1} (1-\alpha)^{\beta -1} d\alpha$.  Here, $\rho >0$ while 
$\beta>-1$ as $d \in \{ 1, \ldots, n \}$. The characteristic index then
depends on the second parameter, $\beta$.  
\begin{enumerate}
\item For $\beta >0$, the characteristic index is
\begin{eqnarray}
\label{hp}
\zeta_n &=&  B(\rho, \beta) - B(\rho+n, \beta) = B(\rho, \beta) \left( 1 - \frac{\rho^{\uparrow n}}{(\rho+\beta)^{\uparrow n}} \right)
\end{eqnarray}
where $\rho^{\uparrow k} = \rho \cdot (\rho + 1) \cdot \cdot \cdot (\rho + k -1)$ 
is the ascending factorial and $B(\rho, \beta)$ is the beta function.  
The previous mentioned example, $q(n-1,1) = 1/n^2$, corresponds to
$\rho = \beta = 1$.

\item For $\beta = 0$, the characteristic index is 
\begin{eqnarray}
\nonumber
\zeta_n &=& \int_0^1 \alpha^{\rho-1}(1 + \alpha + \cdots + \alpha^{n-1})\, d\alpha \\
\label{hp}
	&=& \sum_{j=0}^{n-1} \frac1{\rho + j} = \psi(n+\rho) - \psi(\rho),
\end{eqnarray}
where $\psi$ is the derivative of the log gamma function. 
This process is called the \emph{harmonic process}.

\item Finally, for $\beta \in (-1,0)$, the process has characteristic index 
\begin{align*}
\zeta_n &= \sum_{j=0}^{n-1} \frac{\Gamma(j+\rho) \Gamma (\beta + 1) }{\Gamma ( j + 1 + \beta + \rho)} \\
&= O \left( n^{-\beta} \right)
\end{align*}
\end{enumerate}
\end{example}

A very similar process to the harmonic process with density measure
$dw(\alpha) = \alpha^{\rho-1} d\alpha / (-\log(\alpha))$ 
has characteristic index
\begin{eqnarray}\nonumber
\zeta_n &=& \int_0^1 (1-\alpha^n) \alpha^{\rho-1} \, d\alpha / (-\log\alpha) = \int_0^\infty (1 - e^{-n z}) z^{-1} e^{-\rho z}\, dz \\
\label{gp}
	&=& \log(1+n/\rho).
\end{eqnarray}
This is the characteristic index of the gamma process, which is explored in
more detail in section~3.

\beginignoretext
It is apparent from these examples and the preceding discussion that if $\Psi(t)$
is the characteristic exponent  of a stationary random measure,
i.e.,~$-\Psi(-t)$ is the cumulant generating function
of an infinitely divisible positive random variable,
then $\zeta_n = \Psi(n)$ evaluated at the non-negative integers
is the characteristic index of an exchangeable Markov partial ranking.
Moreover, the associated L\'evy measure $e^{-z} w(e^{-z})\, dz$ for $z > 0$
tells us much about the evolution of the process, the distribution of the number of
blocks and the block sizes, and so on.
The relevance and importance of this construction for application to survival studies
 are described in section~3.
\endignoretext

\subsection{Holding times}
\label{hold_times}
A partial ranking of $[n]$ is not only an ordered partition $(B_1,\ldots,B_k)$,
but also a strictly decreasing sequence of subsets
\[
[n]\equiv R_0 \supset R_1\supset\cdots\supset R_k \equiv\emptyset,
\]
called risk sets,
in which the increments $B_r = R_{r-1} \setminus R_r$ are non-empty.
By associating with each risk set an independent
exponentially distributed holding time,
a Markov process in continuous time is constructed whose
trajectories are non-increasing subsets of $[n]$, with the empty
set as the absorbing state.
Moreover, by choosing the rate function in an appropriate way,
the survival process can be made both consistent under subsampling,
 and exchangeable under permutation of particles.
Consistency under subset selection means that there exists an infinite process,
i.e.,~a survival process for the infinite population,
which implies in turn that the ratio $p_{n+1}/p_n$ determines
the conditional distribution of $T_{n+1}$ given the survival times $T[n]$ for
the initial sample.
It is also possible to compute the conditional distribution
given $T[n]$ and $T_{n+1} > 5$,
i.e.,~to prognosticate in a mathematically consistent manner.

An argument essentially the same as that used in section~4 of 
McCullagh, Pitman and Winkel (2008) leads to the following
consistency condition on the rate function

\begin{proposition}
A sequence of rate functions, $\{ \tau_n \}_{n=1}^\infty$, is consistent
if
\[
\tau_{n+1}(1 - q(n, 1)) = \tau_n.
\]
In other words, the characteristic index $\tau_n \equiv \zeta_n$ is also the exponential failure rate
needed to ensure consistency of the continuous-time Markov process.  
\end{proposition}

While the assumption of temporal homogeneity seems natural for the Markov model, the 
argument can be extended to non-constant holding rates.  That is, we may specify 
a marginal hazard function, $h_2: \Real^+ \to \Real$.  Consistency implies 
the above relation holds with $h_{n} (t)$ replacing $\tau_n$.  Alternatively, a monotone 
transformation of time can be applied in order to modify the marginal distribution.

\subsection{Density function}
Since the evolution of the process $R(t)$ is Markovian, it is a
straightforward exercise to give an expression for the probability
density function at any specific temporal trajectory.
The observation space consists of a partial ranking $B$ of $[n]$
comprising $k=\#B$ disjoint subsets, and for each subset a failure time.
The probability that the first failure occurs in the interval $dt_1$ and
that $B_1$ is the set of failures is
\[
 \zeta(n) e^{- \zeta(n) t_1} \, dt_1 \times q(\Rsharp_1, \Bsharp_1) =
 e^{- \zeta(n) t_1} \, dt_1 \times \lambda(\Rsharp_1, \Bsharp_1),
\]
where $\zeta(n) \equiv \zeta_n$ and  $\lambda(r,d) = q(r,d) \zeta(r+d) = (-1)^{d-1} (\Delta^d \zeta)(r)$.
Continuing in this way, it can be seen that the joint density at
any temporal trajectory $R(\cdot)$ consisting of $k$ blocks with
failure times $0 < t_1 < \cdots < t_k$ is
\begin{equation}\label{density}
f_n(B, t) = \exp\Bigl(- \int_0^\infty \zeta(\Rsharp(s))\, ds \Bigr)
\prod_{j=1}^k \lambda\bigl(\Rsharp(t_j), \Bsharp(t_j)\bigr).
\end{equation}
Here $k = \#B$ is the number of blocks, or more generally the number of
distinct failure times, and
$B_j \equiv B(t_j)$ is the block or subset of particles failing at time~$t_j$.
The density is non-negative, and the integral is one.
In the absence of censoring, this means
\[
\sum_{B\in\OP_n} \int_{0<t_1<\cdots<t_{\#B}} f_n(B, t)\, dt = 1,
\]
so the number of blocks $1\le \#B\le n$ is a random variable whose
distribution is determined by (\ref{density}), and hence by~$\zeta$.
Supplementary figure~\ref{power_simulation} shows simulated survival processes
with characteristic index $\zeta_n = n^\rho$ for various values of $\rho$ together 
with the associated conditional survival function.
Also, supplementary figure~\ref{harmonic_simulation} shows 
several simulated harmonic survival processes for various choices of $\rho >0$
along with the conditional survival function.

The random sequence of failure times $T_1, T_2,\ldots$ whose finite-dimensional joint distributions
are given by~(\ref{density}) is infinitely exchangeable.
The $n$-dimensional joint distribution is continuous in the sense that it has no fixed atoms.
For $n\ge 2$, it is not continuous with respect to Lebesgue measure in $\Real^n$ because
the distribution has condensations on all diagonals implying that
$\pr(T_1 = T_2) = q(0,2) > 0$,
and likewise for arbitrary subsets.
The one-dimensional marginal distributions are exponential with rate $\zeta(1)$.
However, a monotone continuous temporal transformation that sends Lebesgue measure
to the measure $\nu(\cdot)$ on $(0,\infty)$, also transforms (\ref{density})
to an exchangeable semi-Markov process with density
\begin{equation}\label{semi-markovdensity}
f_n(B, t) = \exp\Bigl(- \int_0^\infty \zeta(\Rsharp(s))\, d\nu(s) \Bigr)
\prod_{j=1}^k \lambda\bigl(\Rsharp(t_j), \Bsharp(t_j)\bigr)\, \nu(dt_j).
\end{equation}
If $g(T) = \nu ( ( 0, T))$ is the associated monotone continuous function then 
$g^{-1} (T_i)$ is exponential with rate $\zeta(1)$.

Although the argument leading to (\ref{density}) did not explicitly
consider censoring, the density function has been expressed in integral
form so that censoring is accommodated correctly.
The pattern of censoring affects the evolution of $\Rsharp$,
and thus affects the integral,
but the product involves only failures and failure times.

\subsection{Sequential description}
Since $f_n$ is the marginal distribution of $f_{n+1}$, the conditional distribution
of $T_{n+1}$ given the temporal evolution of the risk set 
$R[n] \equiv (B, t)$ for the first $n$ particles is given by the ratio $f_{n+1}(B',t')/f_n(B, t)$,
where $(B',t')$ is any event compatible with the observation $(B,t)$.
If the new particle fails at one of the previous failure times,
then $t'=t$, and $B'$ is obtained by inserting
the new particle into one of the blocks of $B$;
otherwise the new particle fails interstitially as a singleton in one of the
intervals, $(0, t_1), \ldots, (t_{\#B}, \infty)$,
in which case $\#B' = \#B + 1$.

The conditional distribution is best described
in terms of the conditional hazard measure, $\Lambda$,
which is such that 
\[
\pr(T_{n+1} > t \given R[n]) = e^{-\Lambda((0,t])}.
\]
The conditional hazard has a continuous component supplemented by
an atom at each previously observed failure time.
The continuous component has a density and a cumulative hazard
\begin{eqnarray*}
h(t) &=&  \zeta(\Rsharp(t)+1) -  \zeta(\Rsharp(t))
	= (\Delta\zeta)(\Rsharp(t)), \\
H(t) &=& \int_0^t (\Delta\zeta)(\Rsharp(s))\, ds.
\end{eqnarray*}
Note that $\Rsharp$ is piecewise constant, so the integral is trivial
to compute, but censoring implies that it is not necessarily constant
between successive failures.
For all consistent Markov survival processes, $\zeta_1 > 0$ implies that
the continuous component has infinite total mass,
so $\pr(T_{n+1} < \infty \given R[n]) = 1$,
i.e.,~$T_{n+1}$ is finite with probability one.

Let $t$ be a failure time with
$\Rsharp(t^-) = r+d$ and $\Rsharp(t) = r$ with $d > 0$.
At each such point the conditional hazard has an atom with finite mass
\[
\Lambda(\{t\}) = \log\frac{\zeta(r+d)\, q(r, d)}{\zeta(r+d+1)\, q(r+1,d)},
\]
or, on the probability scale,
\[
\exp(-\Lambda(\{t\})) = \frac{\zeta(r+d+1)\, q(r+1, d)}{\zeta(r+d)\, q(r,d)}
= \frac{(\Delta^d \zeta)(r+1)}{(\Delta^d\zeta)(r)}.
\]
The total mass of the atoms is finite.

Given the trajectory of the risk set for the first~$n$ particles,
the conditional survival function is 
\begin{eqnarray}\label{conditionaldistribution}
\pr(T_{n+1} > t \given R[n]) 
		&=& \exp(-H(t)) \prod_{j : t_j \le t} \frac{(\Delta^{d_j} \zeta)(r_j+1)} {(\Delta^{d_j}\zeta)(r_j)}.
\end{eqnarray}
Although this may look a little complicated, it is not difficult to generate the survival times
sequentially for processes whose characteristic index admits a simple expression
for finite differences.
Right censoring is automatically accommodated by the integral in the continuous component,
so the observed trajectory $R[n]$ may be incomplete.

The harmonic process (\ref{hp}) with characteristic index
$\zeta_n = \nu (\psi(n+\rho)-\psi(\rho))$ for some $\nu > 0$ is such that
$(-1)^{d-1}(\Delta^d\zeta)_r = \nu\Gamma(d)/\ascf{(r+\rho)}d$.
Accordingly, the continuous component of the conditional hazard is
$h(t) = \nu/(\Rsharp(t) + \rho)$, implying that
\[
H(t) = \sum_{i: t_i \le t} \nu \frac{t_{i} - t_{i-1}}{\Rsharp(t_{i-1}) + \rho} 
+ \nu\frac{t - t_{j}}{\Rsharp(t_j)+ \rho},
\]
where the sum runs over event times, censored or failure, such that $t_i \le t$,
and $t_j$ is the last such event.
The discrete component~(\ref{conditionaldistribution}) is a product over failure times
\begin{equation}\label{kaplan-meier}
\prod_{j : t_j \le t}  \frac{(\Delta^{d_j} \zeta)(r_j+1)}{(\Delta^{d_j} \zeta)(r_j)} =
\prod_{j : t_j \le t} \frac{r_j+\rho}{r_j+d_j+\rho}.
\end{equation}
In most cases of practical interest, the continuous component is negligible
over the greater part of the range of interest;
for small $\rho$, the discrete component is essentially the same as
the right-continuous version of the Kaplan-Meier product limit estimator.  Note that exchangeability implies the joint conditional survival probability, $\pr ( T_{n+1} > t, T_{n+2} > t^\prime \given R[n])$, is distinct from the Kaplan-Meier product estimate which assumes independence among individuals.

\subsection{Weak continuity of predictions}
The conditional distribution of $T_{n+1}$ given
the sequence $\bft=(t_1,\ldots, t_n)$ of previous failures
has an atom at each distinct failure time, and is continuous elsewhere.
The predictive distribution is weakly continuous if a small
perturbation of the failure times gives rise to a small
perturbation of the predictive distribution.
In other words, for each $n\ge 1$, and for each non-negative vector~$\bft$,
\[
\lim_{\epsilon\to 0} \pr(T_{n+1}\le x \given T[n]=\bft+\epsilon) =
\pr(T_{n+1} \le x \given T[n]=\bft)
\]
at each continuity point, i.e.,~$x> 0$ not equal to one of the failure times.

Weak continuity is an additivity condition on the 
hazard atom as a function of the risk set and the number of tied failures.
If $d$ failures are exactly tied, the hazard atom is
\[
-\log \frac{(\Delta^d\zeta)(r+1)}{(\Delta^d\zeta)(r)};
\]
the same failures occurring as singletons one millisecond apart
give rise to $d$ distinct hazard atoms whose sum is
\[
-\log \frac{(\Delta\zeta)(r+d)}{(\Delta\zeta)(r)} .
\]
The predictive distribution is weakly continuous if these are equal
for every integer $r\ge 0$ and $d\ge 1$.

\begin{theorem}[Continuity of predictions] \label{cty_pred_thm}
A Markov survival process has weakly continuous predictive distributions
if and only if it is a harmonic process:
$\zeta(n) \propto  \psi(n+\rho) - \psi(\rho)$ for some $\rho > 0$.
The iid exponential model is included as a limit point.
\end{theorem}

To see that it is not satisfied by any other Markov survival process,
it is sufficient to consider a sequence in standard form beginning with
$(\Delta\zeta)_0=1$,
$(\Delta\zeta)_1=  \rho/(\rho+1) < 1$. 
Then the key continuity condition determines the subsequent sequence
$(\Delta\zeta)_r = \rho/(\rho+r)$ in conformity with the harmonic series.
The only exception is the iid exponential process, which arises in the limit
$\rho\to\infty$ in which tied failures occur with probability zero.
All other Markov survival processes, including the gamma process,
have predictive distributions that
are discontinuous as a function of the initial configuration.

\subsection{Self-similarity and lack-of-memory}
Every Markov survival process has the property that
the conditional joint distribution of the residual lifetimes
$T_1 - t,\ldots, T_n - t$ given that $\min(T_1,\ldots, T_n) > t$
is the same as the unconditional distribution of $T_1,\ldots, T_n$.
This property, called \emph{lack of memory}, follows from
the Markov property of the distribution~(\ref{density}).

In general, if $t > 0$ is fixed, only a subset of the
initial sample of $n$ particles will survive beyond that time.
Let $S = \{i \colon T_i > t\} \subset[n]$ be the survivors.
Every Markov survival process with consistent finite-dimensional distributions
$\{p_n\}$ also has the property that the
conditional joint distribution given~$S$ of the residual lifetimes
$\{T_i - t \colon i\in S\}$ is distributed as $p_{S^\sharp}^{}$.
This property, called \emph{self-similarity}, is a consequence of
Markov homogeneity, namely that
the transition intensities are consistent and constant in time.

Lack of memory is a property of the distribution $p_n$ alone,
and does not require consistency of $p_n$ with $p_{n+1}$.
By contrast, self-similarity is a property of the process.
Obviously, self-similarity implies lack of memory.

\subsection{Seeded series \& urn models}
Suppose that an initial sequence consisting of $m$ values $T_1,\ldots, T_m$
is given, and that these correspond to an initial risk-set trajectory $R_0$.
Subsequent values are generated using the Markov rule (\ref{conditionaldistribution}).
The sequence $T_{m+1},\ldots$ is called a seeded series because the joint distribution
depends strongly on the initial sequence.
It is natural to ask what the joint distribution of $T_{m+1},\ldots, T_{m+n}$ is---whether
it is stationary, whether it is exchangeable, and so on.

It is not difficult to see that the joint density of the $m+n$-configuration $R$ given $R_0$
is
\[
\exp\Bigl(-\int_0^\infty \bigl(\zeta(\Rsharp(s)) - \zeta(\Rsharp_0(s)\bigr) \, ds \Bigr)
	\prod_{\rm deaths} \frac{\lambda(\Rsharp(t_j), \Bsharp(t_j))} {\lambda(\Rsharp_0(t_j), \Bsharp_0(t_j))}
\]
where the zero-order difference is defined as $\lambda(r, 0) = \Delta^0\zeta(r) = 1$.
The product runs over all death times.

In other words, regardless of the initial configuration,
the subsequent series is infinitely exchangeable.
Although it is Markovian, the initial series introduces persistent temporal features,
so the failure rate is not temporally homogeneous.
The one-dimensional marginal distributions satisfy (\ref{conditionaldistribution})
(i.e. the distribution of $T_i$ given $R_0$).

Supplementary figure~\ref{harmonic_simulation} illustrates the impact of seeding on conditional
survival function estimates.  Specifically, the $50$ initial values are assumed
independent, uniform on $(1,2)$. An additional $400$ survival 
times are generated conditional on the seeded series given the harmonic process for a given 
choice of parameters $\nu$ and $\rho$.

\subsection{Number of blocks \& block sizes}
The distribution of the number of blocks in a random ordered partition
 depends critically on the splitting probabilities.
The mean number of blocks satisfies the recurrence relation
\[
\mu_n = 1 + \sum_{d=1}^n \binomial n d q(n-d, d) \mu_{n-d}
\]
and there is a similar recurrence relation
\[
M_n(t) = e^t \,\sum_{d=1}^n\binomial n d q(n-d, d) M_{n-d}(t)
\]
for the moment generating functions.
The interest here is in the behaviour for large~$n$.

At one extreme, $q(n-1, 1) = 1/n$ implies $q(r,d) = 0$ for $d > 1$,
and $\mu_n = n$, so every block is a singleton and the partition is trivial.
Similarly, $q(n-1, 1) = 1/(n+\rho)$ implies that the ratio $\mu_n/n$ has a strictly
positive limit $1/(1+\rho)$, implying that a positive fraction of the blocks are singletons.
%
%

As for block sizes, it is clear that the distribution of the size of block~$i$,
$\# B_i$, stochastically dominates the distribution for all subsequent blocks.
By a stick-splitting argument, the expected block size for any 
block can be derived by first examing $E[\# B_1 ]$.  
For $q(n-1,1) = 1/(n+\rho)$, the 
asymptotic expected size of the first block is $(1+\rho)$. 


The beta-splitting rules exhibit different behavior depending on the value of $\beta$.  
Lemma~\ref{block_size_num_thm} provides a description of the 
asymptotic behavior of the number of blocks and block size for this process.

\begin{lemma}
\label{block_size_num_thm}
Let $D_n (\rho,\beta)$ denote the number of blocks given $n$ individuals for the beta process with $\rho \in (0, 
\infty)$ and $\beta \in (-1,\infty)$.  Then
\begin{enumerate}
\item For $\beta > 0$, as $n \to \infty$,
\begin{equation*}
E [ D_n (\rho, \beta) ] \sim \log (n)
\end{equation*}
where $x_n \sim y_n$ if $\lim_{n \to \infty} x_n / y_n = C \in (0, \infty)$ for $x_n$ and $y_n$
two non-negative sequences.

The fraction of edges in the first block, $n^{-1} \# B_{1,n}$, is asymptotically
distributed $Beta(\beta, \rho)$. 
Therefore, the relative frequencies within each block are given by
\[
(P_1, P_2, \ldots) = (W_1, \bar{W}_1 W_2, \bar{W}_1 \bar{W}_2 W_3, \ldots) 
\]
where $W_i$ are independent beta variables with parameters $(\beta, \rho)$, and
$\bar{W}_i = 1-W_i$.

\item For $\beta = 0$, as $n \to \infty$,
\begin{equation*}
E [ D_n (\rho, \beta) ] \sim \log^2 (n)
\end{equation*}
The above relation holds asymptotically for the gamma process as well.  Moreover, the expected number of particles in the first block, $E [ \# B_{1,n} ]$, is $n/( \rho \, \log n )$.
Asymptotically,
\[
\frac{ \log \# B_{1,n} }{\log n } \overset{D}{\to} U
\]
where $U$ has the uniform distribution on $(0,1)$.

\item For $\beta \in (-1,0)$,  as $n \to \infty$,
\begin{equation*}
E [ D_n (\rho, \beta) ] \sim n^{-\beta}
\end{equation*}
The fraction of edges 
in the first block is asymptotically $n^\beta$.  Asymptotically, for all $\rho > 0$,
\[
\pr ( \# B_{1,n} = d ) \to g_{\beta} (d)
\]
where $g_{\beta} (d)$ is defined by
\[
g_{\beta} (d) = \frac{-\beta \cdot \Gamma (d+\beta)}{\Gamma(d+1) \cdot \Gamma( 1+ \beta) } 
=\frac{-\beta}{\Gamma(1+\beta)} d^{\beta-1}
\]
for large $d$. So the number of particles in the first block has a power law distribution of degree $1-\beta$.
\end{enumerate}

\end{lemma}

The proof can be found in Appendix~\ref{app:exp_num_blks}.  

\section{Exchangeable mixture model}
\subsection{Introduction}
The development in this section follows closely
Kalbfleisch (1978), Clayton (1991), and Hjort (1990).
It considers survival processes driven by an arbirary, 
completely independent, stationary random measure,
not necessarily the gamma process.
This section considers only exchangeable processes, but 
these methods can be easily extended to handle inhomogeneity.

\subsection{Completely independent random measure}
A non-negative measure $\Lambda$ on $\Real$ is said to be \emph{completely independent}
if the random variables $\Lambda(A_1),\ldots, \Lambda(A_n)$ are independent whenever 
$A_1,\ldots, A_n $  are disjoint subsets
(Kingman 1993, chapter~8),
and \emph{stationary} if the distribution is unaffected by translation in~$\Real$.
The distribution of $\Lambda(A)$ 
is necessarily infinitely divisible, and the characteristic exponent
\begin{equation}\label{characteristicexponent}
\log E(e^{-t \Lambda(A)}) = -\nu(A)\, \Psi(t)
\end{equation}
for $t \ge 0$ is essentially the same as the cumulant generating function.
The characteristic exponent need not be analytic at the origin,
so, even if the cumulants are finite, 
there need not be a Taylor expansion to generate them.
The measure is stationary if $\nu(A) = \nu \,|A|$ is proportional to
Lebesgue measure on $\Real$, which is subsequently assumed for simplicity of
notation (i.e. $\nu \ge 0$, a constant).
The L\'evy-Khintchine characterization for positive random variables implies
\[
\Psi(t) = \gamma t + \int_0^\infty (1 - e^{-z t}) \, dw(z)
\]
for some $\gamma \ge 0$ and some measure $w$ on $(0,\infty]$, called the L\'evy measure,
such that the integral is finite for $t > 0$.

The random measure has a simple interpretation in terms of a Poisson point process
$X\subset \Real\times(0,\infty)$ with L\'evy intensity $dt\, dw(z)$ such that
\[
\Lambda(A) = \gamma|A| + \int_{A\times (0,\infty)} z\, dX
\]
is Lebesgue measure plus the sum of the $z$-components of
the points of $X$ that fall in $A\times(0,\infty)$.
If the Lebesgue component is missing, i.e.~$\gamma=0$, then
$\Lambda$ has countable support.
If the L\'evy measure is finite, then $\Lambda$ has finitely many atoms in bounded subsets;
otherwise, if $w$ is not finite, $\Lambda$~has countable dense support.

\subsection{Survival process}
Let $\Lambda$ be a stationary, completely independent, random measure on $\Real$,
and let $T_1, T_2,\ldots$ be conditionally independent and identically distributed
random variables such that
\[
\pr(T_i > t \given \Lambda) = \exp(-\Lambda(0, t])
\]
for $t \ge 0$.
In other words, the survival times are non-negative
with cumulative conditional hazard $H(t) = \Lambda(0, t]$.
Conditional on $\Lambda$, the multivariate survival function is
\[
\pr(T_1 > t_1,\ldots, T_n > t_n \given \Lambda) = \exp\Bigl(-\int_0^\infty \Rsharp(s) \,d\Lambda(s) \Bigr),
\]
where $\Rsharp(s) = \#\{i \colon t_i \ge s\}$.
From the definition (\ref{characteristicexponent}) of the characteristic exponent, 
the unconditional joint survival function is
\begin{equation}\label{jsf}
\pr(T_1 > t_1,\ldots, T_n > t_n) = \exp\Bigl(-\nu\int_0^\infty \Psi(\Rsharp(s) )\, ds\Bigr),
\end{equation}
depending on the the characteristic exponent at integer values.
The following connects the survival process associated with completely independent random
measures to the set of Markov survival processes.

\begin{theorem}
The unconditional risk set evolves as a Markov process with characteristic index $\zeta_n =\nu \Psi(n)$ evaluated at the positive integers.
\end{theorem}

\begin{proof}
Given $\Lambda$, the conditional density  at the risk-set trajectory $R$ consisting of
$k\le n$ failure times is
\[
\exp\biggl(-\int_0^\infty \!\Rsharp(t)\, d\Lambda \biggr) \prod_{j=1}^k
	\bigl(1 - e^{-\Lambda(dt_j)}\bigr)^{d_j}
\]
where $d_j \ge 1$ is the number of failures at $t_j$.
Bearing in mind that $\Lambda$ is completely independent and that
$E(e^{-s \Lambda(A)}) = e^{-\nu |A| \Psi(s)}$, the contribution to the unconditional density
of one failure time
with $\Rsharp(t^-) = d+r$ and $\Rsharp(t) = r$ is an alternating sum
\begin{eqnarray*}
E\Bigl( e^{-r \Lambda(dt)}\, \bigl(1 - e^{-\Lambda(dt)})^d \Bigr) &=&
\sum_{j=0}^{d} (-1)^j \binomial {d} j  e^{-\nu \, dt\,\Psi(r+j)} \\
 &=& \nu\, dt\, \sum_{j=0}^{d} (-1)^{j+1}\binomial {d} j \Psi(r+j) + o(dt)\\
&=& \nu\, (-1)^{d-1} (\Delta^d\Psi)(r)\, dt + o(dt).
\end{eqnarray*}
Note that $(-1)^{d-1} (\Delta^d\Psi)(r) \ge 0$ and is decreasing in both arguments.
It follows that the joint density at $R$ is
\begin{equation}\label{jd}
f_n(R) = \nu^k\, \exp\biggl(-\nu \int_0^\infty \Psi\bigl(\Rsharp(s)\bigr) \,ds \biggr)
	\times
\prod_{j=1}^k (-1)^{d_j-1} (\Delta^{d_j}\Psi)(r_j)\, dt_j
\end{equation}
which is the same as (\ref{density}).
\end{proof}

A followup question of interest is whether the converse is also true.  Theorem~\ref{one_to_one_conv}
answers in the affirmative.

\begin{theorem} \label{one_to_one_conv}
Every Markov survival process is generated by a completely independent random measure
\end{theorem}

\begin{proof}
Recall that the set of singleton splitting rules, $\{ q(n,1) \}$, completely determine the splitting rule by consistency.

Equation~(\ref{integral_representation}) provides an integral representation for the singleton splitting rule for all Markov survival processes.
\[
q (n,1) \propto c + \int s^n (1-s) \varpi (ds)
\]

Suppose there were an associated completely independent random measure.  Then the L\'evy-Khintchine characterization implies that the singleton splitting rule is
\begin{align}
q(n,1) &= \nu \cdot \left( \Psi (n+1) - \Psi(n) \right) \\
&= \nu \cdot \left( \gamma + \int_0^\infty e^{-n z} (1- e^{-z} ) w( dz ) \right)
\end{align}
Set $\gamma = c$, and $w(dz) = e^{-z} \varpi ( -\log (s) \in dz )$ and without loss of generality assume $\nu = 1$, as this corresponds to the standardized form of the characteristic index. It rests to check that the measure $w$ is a L\'evy measure.  It is easy to show that
\[
\int_0^1 (1-s) \, \varpi (ds) < \infty \Rightarrow \int_0^\infty (1 - e^{-zt} ) w(dz) < \infty, \hspace{0.3cm} t \geq 0
\]
Therefore the Markov survival process is generated by a completely independent random 
measure.
\end{proof} 

\subsection{Example: homogeneous gamma process}
The density of the gamma distribution $Z\sim \Ga(\nu,\rho)$ at $z > 0$ is
\[
z^{\nu - 1} \rho^\nu e^{- z \rho} / \Gamma(\nu).
\]
For $\nu > 0, \rho > 0$,
the distribution has finite moments of all orders, and moment generating function
\begin{eqnarray*}
M(t) = E(e^{tZ}) &=& \biggl( \frac \rho {\rho-t} \biggr)^\nu \\
K(t) = \log M(t) &=& -\nu\log\bigl( 1 - t / \rho \bigr) = -\nu \Psi(-t)
\end{eqnarray*}
for $t < \rho$.
It follows that the distribution is infinitely divisible with $r$th cumulant $\nu \, \Gamma(r) / \rho^r$. 
It is convenient  to extend the parameter space to include $\nu=\infty$,
in which case $\Ga(\infty, \rho)(\{\infty\}) = 1$ for all $\rho > 0$,
i.e.,~unit mass is assigned to the point at infinity.

The \emph{homogeneous gamma process} with parameter $(\nu, \rho)$ on the real line is a 
stationary random measure $\Lambda$, which has the following properties:
\begin{itemize}
\item{} for each Borel subset $A\subset \Real$ of Lebesgue measure $|A|$,
	the random variable $\Lambda(A)$ is distributed according
	to the gamma distribution $\Ga(\nu |A|, \rho)$;
\item{} the values $\Lambda(A_1),\ldots, \Lambda(A_n)$ assigned by $\Lambda$ to
	disjoint subsets $A_1,\ldots, A_n$ are independent random variables;
\item{} with probability one, $\Lambda$ is purely atomic with countable dense support.
\end{itemize}

To explain the reasoning behind the last point, let $\bfX$ be a Poisson process
with intensity $\nu\, z^{-1}e^{-\rho z}\, dx\, dz$ at $(x, z)$
in the product space $\Real\times(0,\infty)$.
Then
\[
\Lambda(A) = \int_{A\times(0,\infty)} z\, d\bfX
\]
is distributed as the gamma process with parameter $(\nu, \rho)$.
In other words, to each $(x, z) \in \bfX$ there corresponds an atom
of $\Lambda$ such that $\Lambda(\{x\}) = z$.
If $0 < |A| < \infty$, the set $\bfX \cap (A\times(0,\infty))$
is infinite, but the subset
$\bfX \cap (A\times(\epsilon,\infty))$ is finite for every $\epsilon > 0$.
The value assigned by $\Lambda$ to $A\subset\Real$ is the
sum of the $z$-components of all points of $\bfX$ that lie in $A\times(0,\infty)$,
and this sum is finite with probability one.
Since there are no coincident atoms, i.e.,~no pairs of points in $\bfX$
whose $x$-components are equal, the Poisson process is equivalent to $\Lambda$.
For further details, see Kingman (1993, chapter~8).

The characteristic index for the gamma process is $\Psi(t) = \log(1+t/\rho)$,
which means that
\begin{eqnarray}\nonumber
\Delta\Psi(t) &=& \log(\rho+t+1) - \log(\rho+t) \\
\nonumber
(-1)\Delta^2\Psi(t) &=& -\log(\rho+t+2) + 2\log(\rho+t+1) - \log(\rho+t) \\
\nonumber
(-1)^{d-1} \Delta^d\Psi(t) &=& \sum_{j=0}^d (-1)^{j-1} \binomial d j \log(\rho+t+j) \\
\label{gamma_approx}
	&\simeq& \Gamma(d) \bigm/ \ascf{(\rho+t+1/2)} d.
\end{eqnarray}
The joint marginal density of the survival times in the gamma process is
obtained by substitution into (\ref{jd}):
\begin{eqnarray*}
\nu^k\, \exp\biggl(-\nu \int_0^\infty \log\bigl(1+\Rsharp(s)/\rho\bigr) \,ds \biggr)
        \times
\prod_{j=1}^k (-1)^{d_j-1} (\Delta^{d_j}\log)(\rho+r_j)\, dt_j
\end{eqnarray*}
Approximation (\ref{gamma_approx}), which holds with relative error $O(1/t^2)$ for large~$t$,
is the $d$th order forward difference of the characteristic exponent
$\psi(\rho'+t) - \psi(\rho')$,
of the harmonic process with parameter $\rho'=\rho+1/2$.
The discrete component is thus approximately a product over failure times
\[
\prod_{j=1}^k \frac{\Gamma(d_j)} {\ascf{(\rho' + r_j)} {d_j}}
\]
where $d_j$ is the number of failures and $r_j = \Rsharp(t_j)$
is the size of the risk set at time~$t_j$.
In particular, if there are no censored records,
the joint density for the harmonic process with parameter~$\nu,\rho$ is
\[
\frac{\nu^k}{\ascf\rho n}
	\exp\biggl(-\nu \int_0^\infty \bigl(\psi(\rho + \Rsharp(s)) - \psi(\rho)\bigr) \, ds \biggr)
	\prod_{j=1}^k \Gamma(d_j),
\]
where $\sum d_j = n$ and the integral in the exponent may be written as a finite sum
\[
\sum_{k=0}^{n-1} \frac{t_{(n-k)}}{k+\rho}
\]
over the ordered survival times $t_{(1)} \leq t_{(2)} \leq \cdots \leq t_{(n)}$.

\subsection{Parameter estimation} \label{param_est}
Consider the problem of parameter estimation for a two-parameter Markov
survival process with characteristic index of the form
$\zeta_n = \nu \Psi(n)$ where $\Psi(n) = \Phi(n+\rho) - \Phi(\rho)$
for $\nu > 0$, $ \rho > 0$ with $\Phi(\cdot)$ given.
Such a family is generated from a family of L\'evy measures proportional to
$w(dz) e^{-\rho z}$, so $\Lambda$ is expected to have larger atoms if $\rho$ is small.
The gamma and harmonic processes are of this form with $w((0,1))$ and $w((1,\infty))$
both infinite.

The first goal is to estimate the parameters $\rho, \nu$, from observations
$T_1,\ldots, T_n$, some of which may be right censored.
Consistency as $n\to\infty$ is not to be expected, nor is it necessarily important.
Parameter estimation is usually an intermediate step, which is needed primarily
to compute the Bayes estimate, or empirical Bayes estimate, of the survival distribution
$F(t) = \pr(T_{n+1} > t \given\hbox{data})$.

For fixed $\rho$, the survival model (\ref{jd}) is a one-parameter family
with a two-dimensional sufficient statistic
\[
k,\quad\int_0^\infty \Psi(\Rsharp(t))\, dt,
\]
where $1\le k\le n$ is the number of blocks, or more generally
the number of distinct death times.
The first derivative of the log likelihood with respect to $\log\nu$ is
$k - \nu\int\Psi(\Rsharp(t))\, dt$,
the maximum-likelihood estimate is the ratio
\begin{equation}\label{mle}
\hat\nu^{-1} = \frac 1 k \int_0^\infty \Psi(\Rsharp(t))\, dt.
\end{equation}
The Fisher information for $\log\nu$ is $E(k)$,
suggesting that the asymptotic variance of $\log\hat\nu$ is $1/E(k)$.
For the gamma and harmonic processes, $\log(n) < E(k) < n^\epsilon$
implies that the estimator is consistent in the absence of censoring.
but the rate of convergence is very slow.

Estimation of $\rho$ by maximum likelihood is certainly feasible,
but a little more difficult.
One natural option is to consider the product
of the per-particle death rate $\nu\Psi(1)$
 and the total particle time at risk $\int\Rsharp(t)\, dt$, 
and to estimate $\rho$ by setting the product  to the observed number of deaths, i.e.,
\begin{equation}\label{deathrate}
\#\hbox{deaths} = \hat\nu\hat\Psi(1) \int_0^\infty\Rsharp(t)\, dt.
\end{equation}
In the absence of censoring, this is equivalent to setting
the mean survival time $\bar T_n$ to its expected value $1/\hat\zeta_1$.
But $\cov(T_i, T_j) = 1/\zeta_2^2$ for each pair $i\neq j$ implies that
\[
\var(\bar T_n) = 1/(n\zeta_1^2) + (n-1)/(n\zeta_2^2) \to 1/\zeta_2^2
\]
does not tend to zero as $n\to\infty$.
Nonetheless, this second equation
is less sensitive than the first to rounding of survival times,
which is a desirable property for applied work.
The parameter pair can be estimated by iteration.

\subsection{Numerical example}

Consider parameter estimation for a set of failure and censoring times (in weeks) of
the 6-MP subset of leukemia patients taken from Gehan (1965):
\[
6,6,6,6^\star, 7, 9^\star, 10 , 10^\star, 11^\star, 13, 16 , 17^\star, 19^\star, 20^\star, 22, 23, 25^\star, 32^\star, 32^\star, 34^\star, 35^\star
\]
There are $9$ uncensored observations, and a total risk time of $359$ weeks.  
Assuming the survival times are iid exponential with rate parameter, $\theta$, 
then the maximum likelihood estimate of $\theta$ is given by $9/359$, or an expected survival time of $39.89$ weeks.

Consider the two-parameter Markov survival process defined in section~\ref{param_est}, specifically the harmonic and gamma processes. Table~\ref{gehan_table} provides maximum likelihood estimates for $\rho$ and $\nu$. For the gamma process, the empirical Bayes estimate of the rate is then $\hat{\nu} \cdot \log ( 1+ \hat{\rho}^{-1} ) \approx 2.47 \times 10^{-2}$, implying the expected survival time is $40.52$ weeks.  The expected time is the same for the harmonic process.
\begin{table}[!h]
\centering
\caption{ Maximum likelihood estimates for two processes}
\vspace{0.2cm}
\begin{tabular}{r r r c r r}
& \multicolumn{2}{c}{Harmonic process} & & \multicolumn{2}{c}{Gamma process} \\ \cline{2-3} \cline{5-6}
Parameter & Est. & Std. Error & &  Est. & Std. Error\\ \hline
$\rho$ & 21.45 & 19.63 & & 20.95 & 19.61 \\
$\nu$ & 0.53 & 0.44 & & 0.53 & 0.44\\ \hline
\end{tabular}
\label{gehan_table}
\end{table}

Estimation using the maximum likelihood estimate of $\nu$ given $\rho$ and the natural relation between the marginal survival rate associated with the gamma process and $\hat{\theta}$
\[
\hat{\nu} (\rho) \left (\Phi(1+\rho) - \Phi(\rho) \right)= \hat{\theta} = 9/359
\]
yields $(\hat{\rho},\hat{\nu}) = (19.24, 0.49)$
for the gamma process and $(19.73, 0.49)$ for the harmonic process.  Supplementary figure~\ref{prof_lik}
shows that the profile likelihood for $\rho$ is relatively flat for values sufficiently removed from the origin.  For the harmonic process, the figure suggests a $95\%$ confidence interval of approximately $[1.3,5.1]$ for $\log(\rho)$, while under the gamma process, there is an approximate confidence interval of $[1.2, 5.1]$.  Twice the difference in the log-likelihoods at their respective maxima is $8.12 \times 10^{-5}$.

Figure~\ref{gehan_figure} shows the conditional survival distribution given
the observed risk set trajectory.  The empirical Bayes estimate
of the conditional distribution for the harmonic process is approximately equal to that of the 
gamma process.  Both are approximately an average of the Kaplan-Meier product limit estimator and the maximum likelihood exponential estimator of the conditional survival distribution.

\begin{figure}[!h]
\centering
\includegraphics[width = 4in]{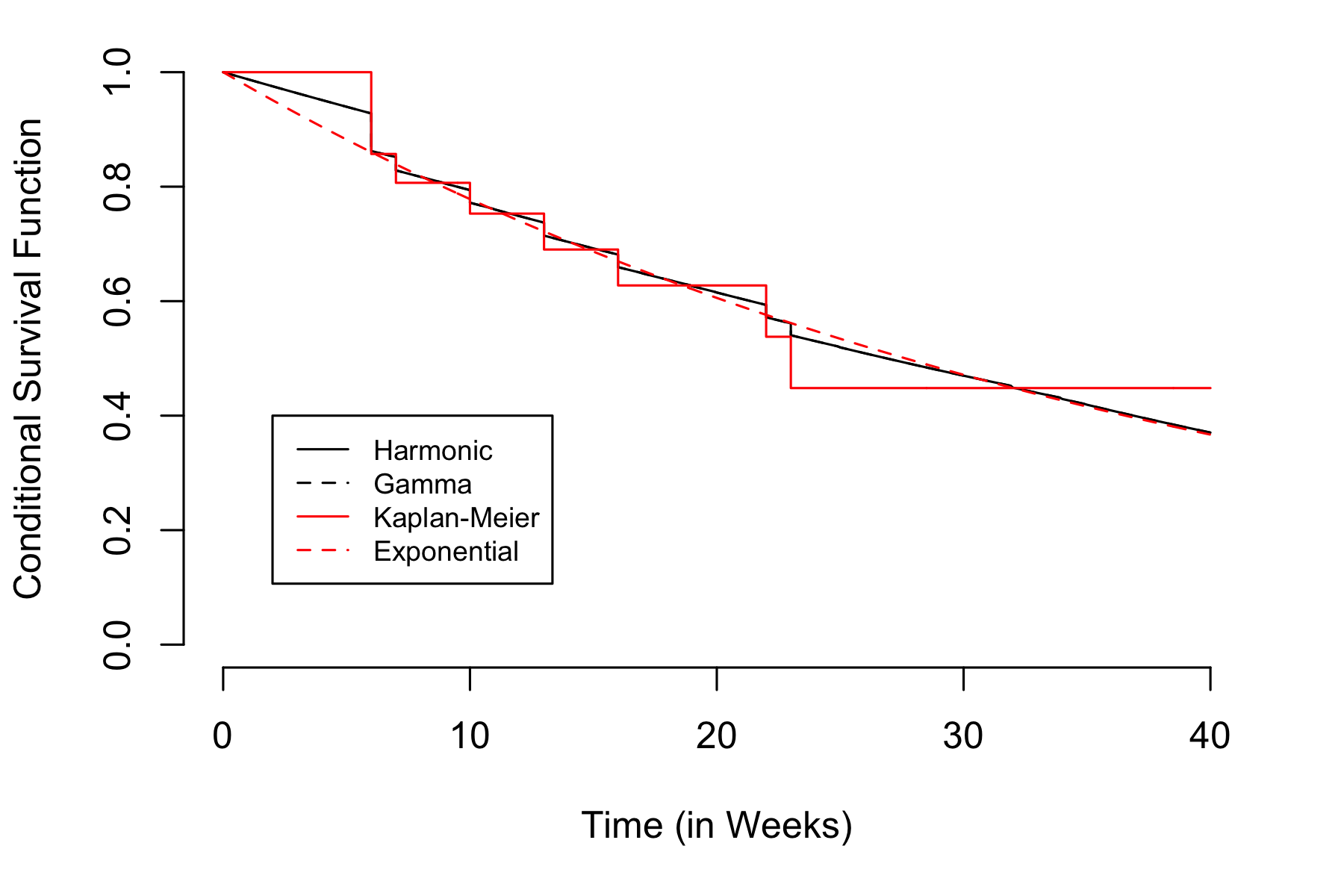}
\caption{Conditional survival distribution for leukemia patients}
\label{gehan_figure}
\end{figure}

\section{Ties as a result of numerical rounding}

Up to this point, individuals having the same recorded survival time
are regarded as failing simultaneously.
Now consider the case where the actual failure times are distinct,
so that tied values arise solely as a result of numerical rounding.
The integral in the exponent is a continuous function of the
risk set trajectory $R(t)$, so an $\epsilon$-perturbation of failure
times has an $O(\epsilon)$ effect on the integral, which is ignored.
However, the remaining term is not a continuous function of the observations,
so an $\epsilon$-perturbation by rounding may have an appreciable
effect on the likelihood.
Most obviously, the statistic $k$, the number of distinct failure times,
is not continuous as a function of~$T$;
if ties are an artifact of rounding, then $k$ is the total
number of failures.

While the likelihood and parameter estimation are affected by ties as a result of numerical rounding, the conditional survival distribution for the harmonic process given $\rho$ and $\nu$ is unaffected due to the weak continuity of predictive distributions.  This suggests it may be best to regard $\rho$ as a fixed ``tuning parameter''.  As all other processes have discontinuous predictive distributions, the use of the harmonic process in applications where ties are likely the result of numerical rounding seems most natural.

\section{Conclusion}

Presented here is the set of exchangeable, consistent Markov survival processes defined by their characteristic index, $\zeta$.  These processes are in correspondence with the set of completely independent random measures, each determined by the characteristic exponent.  

Markov survival processes exhibit multiple tied failure times, generating exchangeable partial rankings among individuals.  It has been shown that the number of blocks depends on the characteristic index.  The harmonic and gamma processes are studied in detail, showing the number of blocks grows asymptotically at a rate of $\log^2 (n)$. Censoring is easily incorporated and parameter estimation given a set of survival and censoring times is examined.  The impact of ties as the result of numerical rounding concludes the discussion.

\newpage

\appendix

\section{Supplementary Figures}

\begin{suppfigure}[!h]
\begin{center}
\includegraphics[height=15cm,width=12.5cm, angle=0]{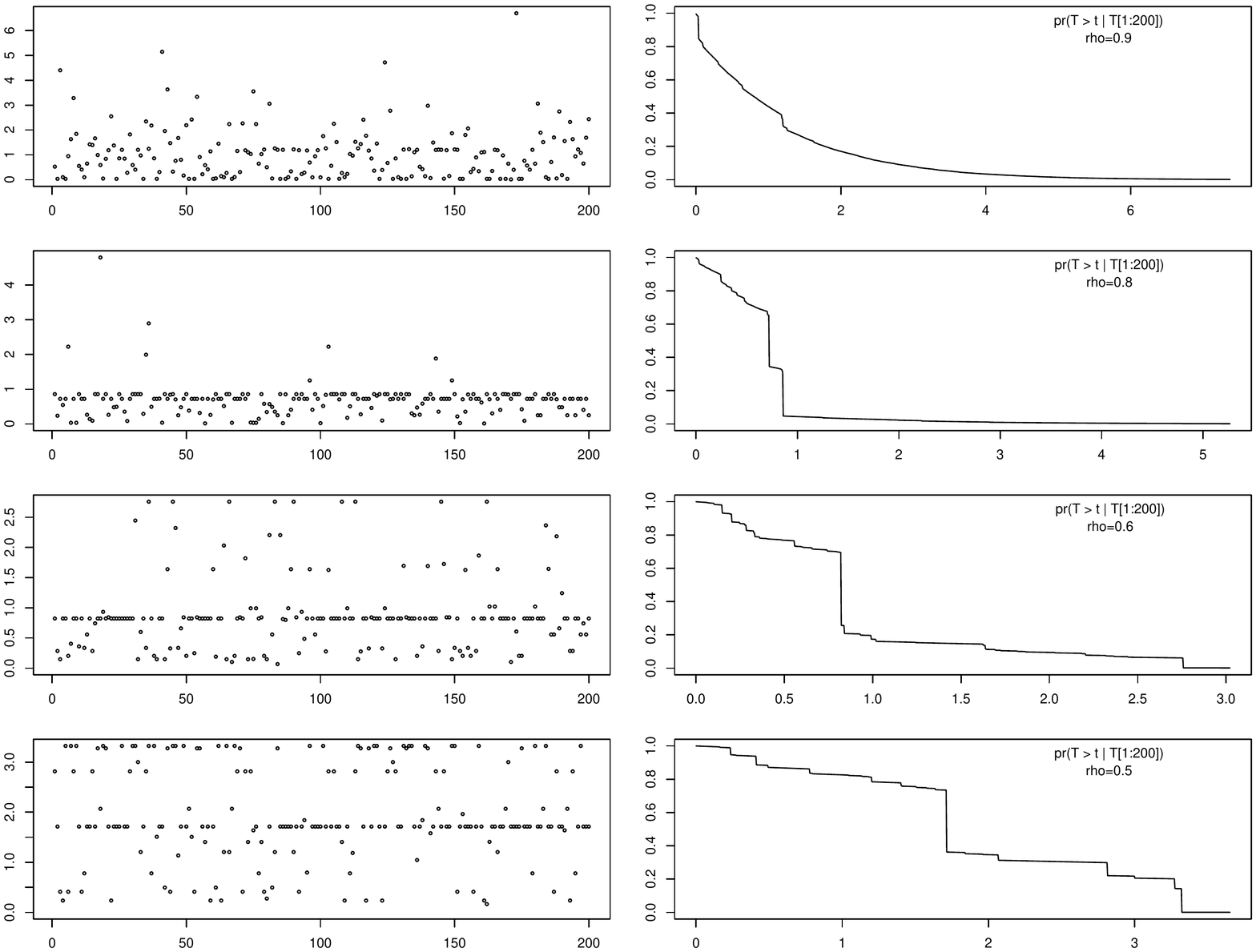}
\end{center}
\vspace{-0.5cm}
\caption{Simulated survival process with index $\zeta_n \propto n^\rho$,
together with the conditional survival function (right panel).
Parameter values: $\rho=0.9, 0.8, 0.6, 0.5$.
Among the first 200 failure times, the number of distinct values
was 150, 65, 42, 25, respectively, but these are highly variable from run to run.}
\label{power_simulation}
\end{suppfigure}

\newpage

\begin{suppfigure}[!h]
\begin{center}
\includegraphics[height=15cm,width=12.5cm, angle=0]{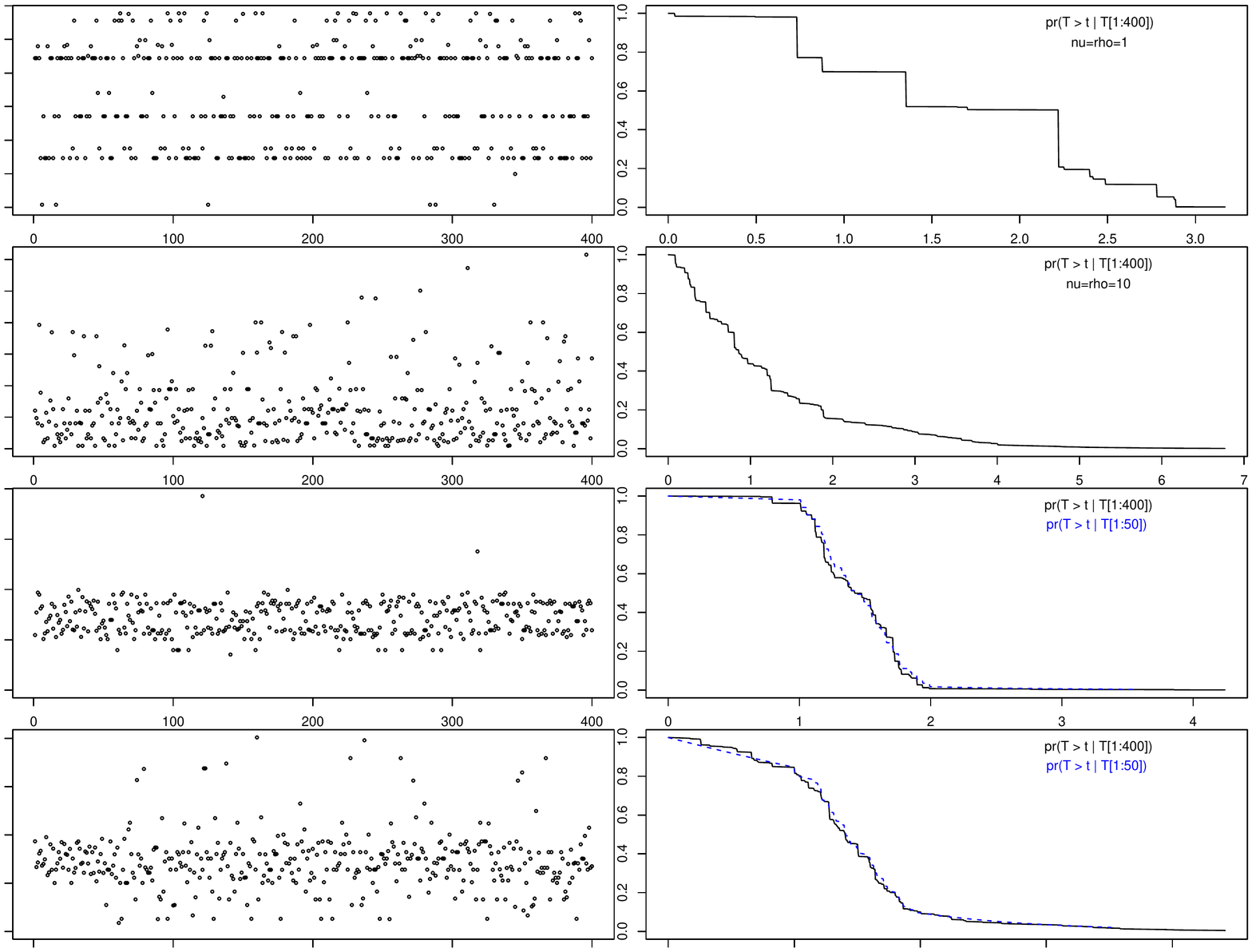}
\end{center}
\vspace{-0.5cm}
\caption{Simulated harmonic survival process, together with
the conditional survival function (right panel).
The lower two series are seeded with 50 initial values, independent, uniform on $(1, 2)$.
Parameter values: $\nu=\rho=1$ in rows 1, 3; $\nu=\rho=10$ in rows 2,~4.
Among the first 400 failure times, the number of distinct values
was 15, 93, 58, 110, respectively, including the initial seeds.}
\label{harmonic_simulation}
\end{suppfigure}

\newpage

\begin{suppfigure}[!h]
\centering
\caption{Profile Log-Likelihood of $\log(\rho)$}
\includegraphics[width = 4in]{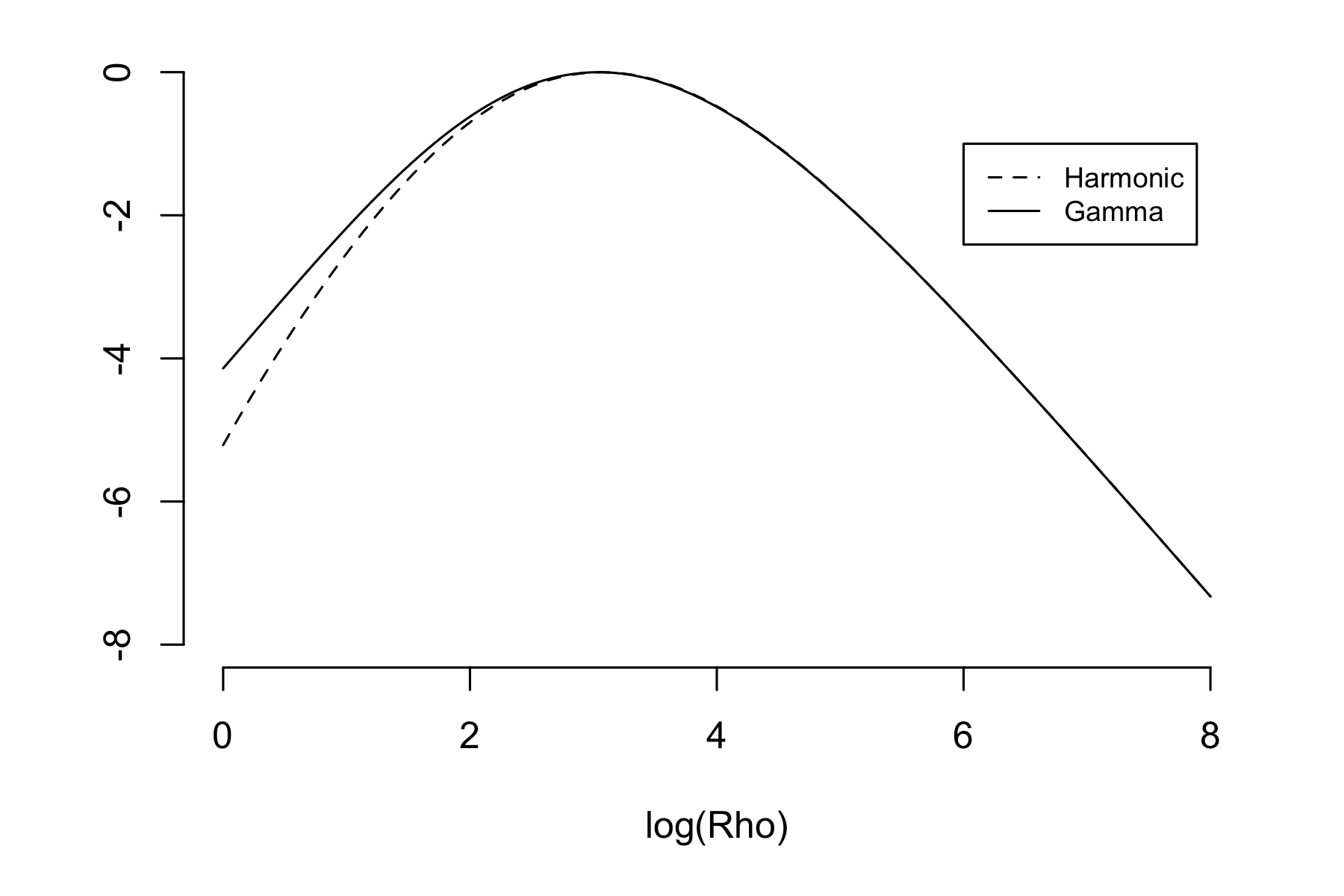}
\label{prof_lik}
\end{suppfigure}

\begin{suppfigure}[!h]
\centering
\caption{Profile log-likelihood of $\beta_1$}
\includegraphics[width = 4in]{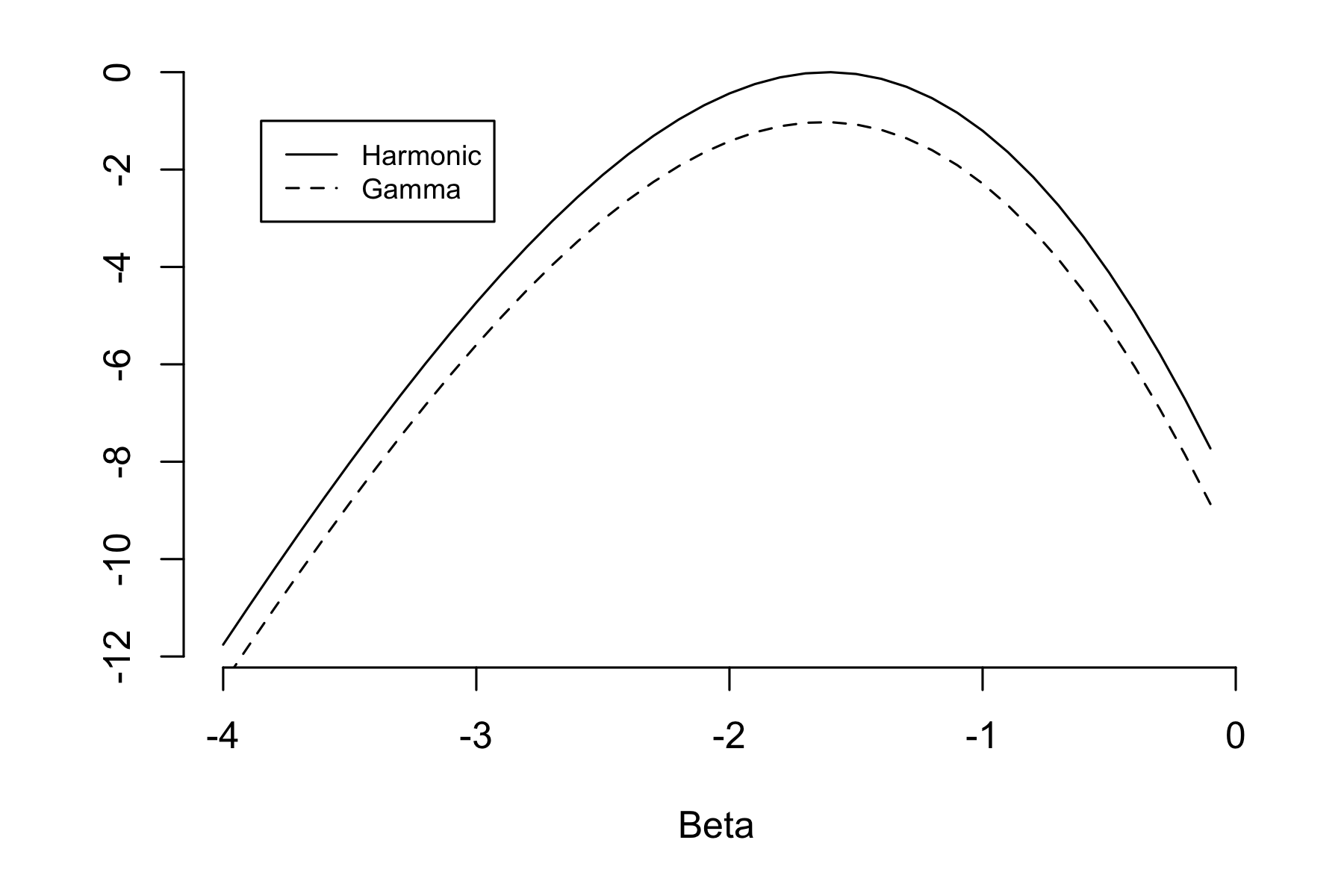}
\label{beta_profile}
\end{suppfigure}

\newpage 

\section{Proof of Proposition~\ref{char_index_prop}}
\label{app:char_index_proof}

\begin{proof}
A consistent splitting rule is completely determined by the set of singleton splitting rules, 
$\{ q(n,1) \given n \ge 0 \}$ where $q(0,1) = 1$ by definition.  
It is immediately apparent from the construction that
$q_{r,1} = \Delta\zeta_r / \zeta_{r+1} = 1 - \zeta_r/\zeta_{r+1}$.
Considering the standardized sequence and using this relation
yields a one-to-one correspondence between the set of singleton splitting rules 
and the characteristic index.  Therefore the characteristic index completely 
determines a splitting rule.  

Showing equation~(\ref{zetacondition}) yields a splitting rule that satisfies 
the consistency condition~(\ref{consistencycondition_eq})
completes the proof. The left hand side is given by
\begin{align*}
(1- q(n,1)) q(n-d,d) &= \frac{\zeta_{n}}{\zeta_{n+1}} \frac{(-1)^{d-1} (\Delta^d \zeta)_{n-d} } { \zeta_{n}} \\
&=  \frac{(-1)^{d-1} (\Delta^d \zeta)_{n-d} } { \zeta_{n+1}}
\end{align*}
The right hand side is given by
\begin{align*}
q(n-d+1,d)+q(n-d,d+1) &= \frac{(-1)^{d-1} (\Delta^d \zeta)_{n-d+1} } { \zeta_{n+1}} + \frac{(-1)^{d} (\Delta^{d+1} \zeta)_{n-d} } { \zeta_{n+1}} \\
&= \frac{(-1)^{d-1} } { \zeta_{n+1}} \cdot \left[ (\Delta^d \zeta)_{n-d+1} - (\Delta^{d+1} \zeta)_{n-d} \right] \\
&= \frac{(-1)^{d-1} } { \zeta_{n+1}} \cdot (\Delta^d \zeta)_{n-d}
\end{align*}
So the splitting rule automatically satisfies the consistency condition.  By definition, the splitting rule must be 
non-negative, and therefore the sequence defines a consistent splitting rule if the forward differences are non-negative,
$(-1)^{d-1} ( \Delta^d \zeta)_r \geq 0$ for all $r \geq 0$, $d \geq 1$.
\end{proof}

\section{Proof of Theorem~\ref{block_size_num_thm} }
\label{app:exp_num_blks}

\begin{proposition}
Define $D_n (\rho)$ denote the number of blocks given $n$ individuals and $\rho \in (0, \infty)$.  As $n \to \infty$,
\begin{equation*}
E [ D_n (\rho) ] \sim \log^2 (n)
\end{equation*}
for both the harmonic and gamma processes, where the ratio tends to a constant equal to $\frac{1}{2 \cdot \Psi_1 (\rho)}$ where $\Psi_1$ is the trigamma function.
\end{proposition}

\begin{proof}
For the gamma process, the recurrence
relation of the expected number of blocks is given by:
\begin{align*}
\mu_n = E[ D_n ] &= 1 + \sum_{d=1}^n {n \choose d} q (n-d, d) \cdot \mu_{n-d} \\
&= 1 + \sum_{d=1}^n {n \choose d} \frac{ (-1)^{d-1} (\Delta^d \zeta)_{n-d}}{\log(1+\frac{n}{\rho})} \cdot \mu_{n-d}
\end{align*}
which gives the following approximation:
\begin{align*}
(-1)^{d-1} (\Delta^d \zeta)_{n-d} &= \sum_{i = 0}^{d-1} {d-1 \choose i} (-1)^{i} \log \left( 1 + \frac{1}{R+i + \rho} \right) \\
&\approx \sum_{i = 0}^{d-1} {d-1 \choose i} (-1)^{i} \frac{1}{R+i + \rho}\\
&= \frac{ \Gamma (d) \Gamma (n - d + \rho ) }{\Gamma (n + \rho ) }
= \frac{ \Gamma(d) } { (n-d + \rho)^{\uparrow d}}
\end{align*}
This is the exact expression for the harmonic process.  So for both processes,
the above implies that the probability of $d$ individuals in block~$1$ is proportional to:
\begin{align*}
{n \choose d} (\Delta^d \zeta)_{n-d} &= \frac{\Gamma(n+1)}{\Gamma(n+\rho)} \cdot  \frac{\Gamma(n-d+\rho)}{\Gamma(n-d+1)} \cdot \frac{1}{d} \\
&\approx \left( \frac{n-d}{n} \right)^{\rho-1} \cdot \frac{1}{d}
\end{align*}
where the approximation
\begin{equation*}
\frac{\Gamma(x+1)}{\Gamma (x + \rho)} \approx x^{\rho-1}
\end{equation*}
is used. Substition of this into the recurrence relation yields:
\begin{equation*}
0 \approx 1 + \frac{1}{\log \left( 1+\frac{n}{\rho} \right)} \sum_{d=1}^n \left( \frac{n-d}{n} \right)^{\rho-1} \cdot \frac{1}{d} \left( \mu_{n-d} - \mu_n \right)
\end{equation*}
Writing $\mu_n \sim c \log^2 (n)$ and approximating the sum by an integral yields:
\begin{align*}
0 &\approx 1 + \frac{c}{\log \left( 1+\frac{n}{\rho} \right)} \int_0^1 \frac{\left( 1-x \right)^{\rho-1} \log ( 1- x) }{x} \left( 2 \log(n) - \log(1-x) \right) dx \\
&\to 1 + 2 \cdot c \cdot \int_0^1 \frac{\left( 1-x \right)^{\rho-1} \log ( 1- x) }{x}  dx \\
&= 1 - 2 \cdot c \cdot \Psi_1 (\rho)
\end{align*}
Giving us $c = \frac{1}{2 \cdot \Psi_1 (\rho)}$ as required.
\end{proof}

A similar argument can be used to obtain the asymptotic expected fraction of particles in 
the first block.

\begin{proposition}
The expected number of particles in the first block, $E [ \# B_1 ]$, is $n/(\rho \log n )$ for 
the harmonic and gamma process.  Asymptotically, for integer values of $\rho$,
\[
\frac{ \log B_{1} }{\log n } \overset{D}{\to} U
\]
where $U$ has the uniform distribution on $(0,1)$.
\end{proposition}

\begin{proof}
The expected number of 
particles in the first block is given by
\begin{align*}
\sum_{d=1}^n d {n \choose d} q(n-d,d) &= \frac{1}{\log ( 1+ \frac{n}{\rho} ) } \sum_{d=1}^n d \left( \frac{n-d}{n} \right)^{\rho -1}  \frac{1}{d} \\
&\to  \frac{n}{\log ( 1+ \frac{n}{\rho} ) } \int \left( 1 - x \right)^{\rho -1}  dx \\
&\approx  \frac{n}{\log(n) } \frac{1}{\rho}
\end{align*}
So the fraction of particles in the first block is roughly $ (\rho \log(n) )^{-1}$.  
For $z \in (0,1)$, the asymptotic distribution is given by
\begin{align*}
\pr \left[ \frac{ \log B_{1} }{\log n }  \leq z \right] &= \frac{1}{\log (n)} \int_{1/n}^{n^{z-1}} x^{-1} (1-x)^{\rho -1} dx \\
&= \frac{1}{\log n} \left[ \sum_{j=0}^{p-1} c_j x^{j} + \log (x) \right]_{1/n}^{n^{z-1}} \\
&\to \frac{1}{\log n} \left[ (z-1) \log (n) - \log (1/n) ) \right] = z \\
\end{align*}
where the second line is true for $\rho \in \mathbb{Z}$.  The result holds for general $\rho>0$ by the squeezing theorem.
\end{proof}

We now consider the number of blocks and block sizes for the general beta-splitting rules.

\begin{proposition}
Define $D_n (\rho, \beta)$ denote the number of blocks given $n$ individuals for the beta process.  As $n \to \infty$,
\begin{equation*}
E [ D_n (\rho, \beta) ] \sim \log (n)
\end{equation*}
where the ratio tends to a constant $c = \frac{1} {\psi ( \rho + \beta) - \psi(\rho) )}$. The fraction of edges 
in the first block, $\# B_1$, is distributed $Beta( \rho, \beta)$.  Therefore, the relative frequencies
within each block is given by
\[
(P_1, P_2, \ldots) = (B_1, \bar{B}_1 B_2, \bar{B}_1 \bar{B}_2 B_3, \ldots) 
\]
where $B_i$ are independent beta variables with parameters $(\rho, \beta)$.
\end{proposition}

\begin{proof}
The probability of $d$ individuals in the first block is given by
\[
{n \choose d} q(n-d,d)  \approx \frac{1}{Z_n} (1-x)^{\rho-1} x^{\beta-1}
\]
where $Z_n = \int_0^1 (1-s^n) s^{\rho-1} (1-s)^{\beta-1} ds = B(\rho,\beta) - B(\rho+n,\beta)$.
As $n \to \infty$, for $\rho, \beta > 0$ the normalization constant converges to $B(\rho, \beta)$. 
Therefore, the expected number of blocks is given by 
\[
0 = 1 + \frac{1}{Z_n} \sum_{d=1}^n \left( \frac{n-d}{n} \right)^{\rho-1} \left( \frac{d}{n} \right)^{\beta-1} n^{-1} \left( \mu_{n-d} - \mu_n \right)
\]
Writing $\mu_n \sim c \log (n)$ then approximating the sum by an integral yields
\begin{align*}
0  &\approx 1 + \frac{c}{B(\rho, \beta)} \int_0^1 (1-x)^{\rho - 1} x^{\beta - 1} \log (1 - x) dx \\
&= 1 + c \cdot E [ \log (1-X) ] \\
&= 1 + c \cdot ( \psi(\rho) - \psi ( \rho + \beta) )
\end{align*}
as desired.  The probability of $d$ out of $n$ particles in block one is given by
\[
\frac{n}{Z_n} \left( \frac{n-d}{n} \right)^{\rho-1} \left( \frac{d}{n} \right)^{\beta-1} n^{-1} 
\to \frac{n}{B(\rho, \beta)}  (1-x)^{\rho - 1} x^{\beta - 1} dx
\]
So the fraction of particles in block one is distributed beta with parameters $(\rho, \beta)$.
\end{proof}

The final case is when $\beta \in (-1,0)$.  Then the number of blocks grows polynomially
in $n$.

\begin{proposition}
Define $D_n (\rho, \beta)$ denote the number of blocks given $n$ individuals for the beta process.  For $\beta \in (-1, 0)$, as $n \to \infty$,
\begin{equation*}
E [ D_n (\rho, \beta) ] \sim n^{-\beta}
\end{equation*}
where the ratio tends to a constant $c = -\Gamma ( \rho + \beta + 1) / ( \Gamma ( \rho) \beta^2 )$. The fraction of edges 
in the first block is asymptotically proportional to $n^\beta$. Asympotically,
\[
\pr [ B_1 = d ] = \frac{-\beta \cdot \Gamma (d+\beta)}{\Gamma(d+1) \cdot \Gamma( 1+ \beta) } 
=\frac{-\beta}{\Gamma(1+\beta)} d^{\beta-1}
\]
for large $d$.
\end{proposition}

\begin{proof}
For $\beta \in (-1,0)$, the characteristic index is given by
\begin{align*}
\zeta_n &= \int_0^1 (1-s^n) s^{\rho-1} (1-s)^{\beta-1} ds \\
&= \int_0^1 \sum_{j=0}^{n-1} s^{j + \rho-1} (1-s)^{\beta} ds \\
&= \sum_{j=0}^{n-1} \frac{\Gamma(j+\rho) \Gamma (\beta + 1) }{\Gamma ( j + 1 + \beta + \rho)} \\
&\approx \frac{\Gamma(1+\beta)}{-\beta} n^{-\beta}
\end{align*}
for $n$ large.
Plugging this into the recursive formula, while 
assuming $\mu_n \sim c n^{-\beta}$ yields
\begin{align*}
0 &= 1 + c \frac{- \beta \, n^{\beta}}{\Gamma(1+\beta) } \sum_{d=1}^n \left( \frac{n-d}{n} \right)^{\rho-1} \left( \frac{d}{n} \right)^{\beta-1} n^{-1} \left( (n-d)^{-\beta} - n^{-\beta} \right) \\
&= 1 - \frac{c \, \beta }{\Gamma(1+\beta) } \sum_{d=1}^n \left( \frac{n-d}{n} \right)^{\rho-1} \left( \frac{d}{n} \right)^{\beta-1} n^{-1} \left( (1-\frac{d}{n})^{-\beta} - 1 \right) \\
&= 1 - \frac{c \; \beta^2 }{\Gamma(1+\beta) } \sum_{d=1}^n \left( \frac{n-d}{n} \right)^{\rho-1} \left( \frac{d}{n} \right)^{\beta} n^{-1}  \\
&\to 1 - \frac{c \; \beta^2 }{\Gamma(1+\beta) } \int \left( 1-x \right)^{\rho-1} \left( x \right)^{\beta} dx \\
&= 1 + c \frac{\beta^2 \; \Gamma (\rho)}{\Gamma(\rho+\beta+1) }
\end{align*}
where for large $n$ the approximation $(1-\frac{d}{n})^{-\beta} = 1 + \beta \frac{d}{n}$ is used.  
The result is immediate.  The fraction of edges in the first block is given by
\[
\frac{- \beta \, n^{\beta}}{\Gamma ( 1+ \beta) } \sum_{i=1}^n \left ( \frac{d}{n} \right)^{\beta} \left( \frac{n-d}{n} \right)^{\rho-1}  n^{-1} = \frac{- \beta \, n^{\beta}}{\Gamma ( 1+ \beta) } \int \left( 1-x \right)^{\rho-1} \left( x \right)^{\beta} dx = \frac{- \beta \, \Gamma (\rho)}{\Gamma(\rho+\beta+1) } n^\beta 
\]
Finally, the asymptotic distribution comes from
\begin{align*}
\pr ( B_1 = d ) &= \zeta_n^{-1} \frac{ \Gamma(d+\beta) \Gamma(n-d + \rho)}{\Gamma(n + \rho + \beta)} \cdot {n \choose d} \\
&\sim \frac{-\beta \cdot n^{\beta}}{\Gamma(1+\beta)} \cdot \frac{\Gamma(n+1)}{\Gamma(n+\rho+\beta)} \frac{\Gamma(d+\beta)}{\Gamma(d+1)} \frac{ \Gamma(n-d + \rho)}{\Gamma(n-d+1)} \\
&\sim \frac{-\beta }{\Gamma(1+\beta)} \frac{\Gamma(d+\beta)}{\Gamma(d+1)}  \left( 1 - \frac{d}{n} \right)^{\rho -1} \\
&\to \frac{-\beta }{\Gamma(1+\beta)} \frac{\Gamma(d+\beta)}{\Gamma(d+1)} 
\end{align*}
as $n \to \infty$.
\end{proof}

\end{document}